\documentclass[11pt,a4paper]{amsart}
\usepackage[english]{babel}
\usepackage[T1]{fontenc}
\usepackage[utf8]{inputenc}
\usepackage{lmodern}
\usepackage{microtype}
\usepackage{amsmath,amsfonts,amssymb}
\usepackage[pdftex]{hyperref}
\usepackage{setspace}
\usepackage[all]{xy}

\title[The inner product of complex forms]
{The inner product on exterior powers\\ of a complex vector space} 
\author{Gunnar \TH\'{o}r Magn\'{u}sson}
\email{gunnarthormagnusson@gmail.com}

\def\^#1{^{[#1]}}
\def\bw#1{\bigwedge{}^{\mkern-5mu #1}\mkern2mu}
\def\I{\mathbf{I}}
\DeclareMathOperator{\tr}{tr}
\DeclareMathOperator{\id}{id}
\DeclareMathOperator{\Ric}{Ric}
\DeclareMathOperator{\End}{End}

\def\RR{\mathbb{R}}

\def\curv{\frac{i}{2\pi} \Theta}
\def\la{\langle}
\def\ra{\rangle}

\newtheoremstyle{slthm}{6pt}{6pt}{\slshape}{0pt}{\bf}{}{1em}{}
\newtheoremstyle{sldef}{6pt}{6pt}{}{0pt}{\bf}{}{1em}{}
\newtheoremstyle{slrem}{6pt}{6pt}{}{0pt}{\slshape}{\hskip1ex---}{1ex}{}
\theoremstyle{slthm}
\newtheorem{theo}{Theorem}[section]
\newtheorem{prop}[theo]{Proposition}
\newtheorem{coro}[theo]{Corollary}

\theoremstyle{sldef}
\newtheorem{exam}[theo]{Example}
\theoremstyle{slrem}
\newtheorem*{rema}{Remark}
\numberwithin{equation}{section}

\makeatletter
\def\eqalign#1{%
 \null\,\vcenter{\openup\jot\m@th
  \ialign{\strut\hfil$\displaystyle{##}$&$\displaystyle{{}##}$\hfil
      &&\hfil$\displaystyle{##}$&$\displaystyle{{}##}$\hfil\crcr#1\crcr}}\,}
\makeatother

\begin{document}

\begin{abstract}
We give a formula for the inner product of forms on a Hermitian vector
space in terms of linear combinations of iterates of the adjoint of the
Lefschetz operator.  As an application, we reprove the
Kobayashi--L\"{u}bke inequality for Hermite--Einstein bundles. 
\end{abstract}

\maketitle

\section*{Introduction}

Let $V$ be a complex vector space of dimension $n$, equipped with a
Hermitian inner product $h$ whose positive $(1,1)$--form we denote by
$\omega = -\operatorname{Im} h$. The inner product on $V$ induces an inner
product on the exterior algebra $\bw{*} V^*$. If we denote the Hodge star
operator by $*$, then this inner product is also defined~by%
$$
\langle u, v\rangle \, \omega\^n 
= u \wedge *\overline v,
$$ 
where $\omega\^k := \omega^k / k!$ for $k \geq 0$ and $u,v$ are elements
of the exterior algebra.

There are two cases where we can easily calculate this inner product
without writing $u$ and $v$ in local coordinates and painfully calculating
minors of the resulting matrices: If $v$ is a primitive
$(p,q)$--form, then
$$
\langle u, v\rangle \, \omega\^n 
= u \wedge \overline v \wedge \omega\^{n-p-q};
$$ 
and if $u$ and $v$ are $(1,1)$--forms then we eventually find that
$$
\langle u, v\rangle \, \omega\^n 
= - u \wedge v \wedge \omega\^{n-2}
+\Lambda u \wedge \overline{\Lambda v} \wedge \omega\^{n}
$$
by decomposing the forms into primitive components. These formulas are 
no better for explicit calculations than the ones involving minors of
matrices, but they come in handy when calculating things that let the
inner product $\omega$ vary and give amusing matrix identities when
written in an orthonormal~basis.%

In this note we generalize the above formulas to arbitrary $(p,q)$--forms
$u$ and $v$. In Theorem~\ref{thm:norm} we show that there exist integers
$b_l$, independent of the vector space $V$, such that
$$
    (-1)^{k(k+1)/2}
    \langle u, v \rangle \, \omega\^n
    = \sum_{l=0}^n (-1)^{l} b_l \, 
    \Lambda\^l u \wedge \Lambda\^l \overline{\I v}
    \wedge \omega\^{n-k+2l}
$$
for any $k$--forms $u$ and $v$. The integers $b_l$ have actually been
studied in quite a different context \cite{Carlitz,Riordan,OEIS}; they are
the coefficients in the series expansion of the reciproqual of a Bessel
function of the first kind. 

The proof of this result is mostly by formal calculations with the
Lefschetz operator and its adjoint on the exterior algebra of $V$, the
only slight difficulty is showing that the coefficients $b_l$ are the
indicated ones. We do this in Section~\ref{sec:on}. There I go into
considerable detail in all calculations, which I hope the reader will
forgive me, but I decided it was best to make all the steps explicit for
the sake of error-checking.  In Section~\ref{sec:tw} we apply our results to
calculate the norm of a curvature form of a Hermitian vector bundle and
reprove the Kobayashi--L\"{u}bke inequality for Hermite--Einstein vector
bundles. The resulting proof is very likely the same as the usual
differential-geometric one if we write all the calculations in local
coordinates.

\begin{rema}
We state and prove all our results on a
finite-dimensional vector space $V$ equipped with a Hermitian inner
product. However, since our calulations are formal the same proofs work 
verbatim on a vector space $V$ equipped with a representation of
$\mathfrak{sl}(2)$, such as the cohomology algebra of a compact
K\"ahler manifold or a projective variety over a field $k$.
\end{rema}

\section{Inner products of exterior forms}
\label{sec:on}

Let $V$ be a complex vector space of dimension $n$ and $\omega$ a Hermitian
inner product on $V$. The Hodge star operator of $\omega$ is $*$,
the Lefschetz operator is $L$ and its adjoint is $\Lambda$. We write
$\I = \sum_{p,q} i^{p-q} \pi_{p,q}$, where $\pi_{p,q} : \bw{*} V^* \to
\bw{p,q} V^*$ is the orthogonal projection.

Recall that a $k$--form $u$ on $V$ is \emph{primitive} if $\Lambda u = 0$.
This is equivalent to $L^{n-k-1}u = 0$. Any $k$--form $u$ on $V$ admits a
\emph{primitive decomposition} $u = \sum L^{k-j} u_j$, which is an
orthogonal decomposition of $u$ where each form $u_j$ is a primitive
$(k-2j)$--form.

If $A$ is an element of an algebra then we define $A\^k = A^k / k!$ for $k
\geq 0$. This entails that
$$
A\^j \cdot A\^k = \tbinom{j+k}{j} A\^{j+k}.
$$
We will use this convention for the element $\omega$ of the exterior
algebra $\bw{*} V^*$ and the operators $L$ and $\Lambda$ on that
algebra.

Consider the sequence of integers defined recursively by $b_0 = 1$ and
\begin{equation}
\label{bezel}
\sum_{l=0}^{p} (-1)^{l} \tbinom{p}{l}^2 b_l = 0
\end{equation}
for $p \geq 1$. This is sequence number A000275 in the \textsl{On-line
encyclopedia of integer sequences}~\cite{OEIS}; see also \cite{Carlitz}
and \cite{Riordan} for not unrelated information about the sequence.
Its first few values are: 
$$
1 \quad 1 \quad 3 \quad 19 \quad 211 \quad 3651 \quad
90.921 \quad 3.081.513 \quad 136.407.699 \quad 7.642.177.651
$$
Our objective in this section is to prove the following theorem.

\begin{theo}
    \label{thm:norm}
Let $u$ be a $k$--form on an $n$--dimensional complex vector
space $V$. Then
$$
    (-1)^{k(k+1)/2}
    \lvert u \rvert^2 \omega\^n
    = \sum_{l=0}^n (-1)^{l} b_l \, 
    \Lambda\^l u \wedge \Lambda\^l \overline{\I u}
    \wedge \omega\^{n-k+2l}.
$$
\end{theo}

We will actually prove this result for a $(p,q)$--form $u$ with $p \leq q$
and $p+q = k$. By conjugation the restriction $p \leq q$ is irrelevant and
it is just a matter of basic combinatorics and degree reasoning to see
that if the result holds for $(p,q)$--forms with $p + q = k$ then it also
holds for $k$--forms, i.e., that the right-hand side respects the
orthogonal decomposition of $\bw{k}V^*$.  Proving this result demands a
certain amount of preparation, all of which rests on the following formula.

\begin{prop}[{{\cite[Proposition~1.67]{Huy}}}]
Let $u$ be a primitive $(p,q)$--form on $V$. Set $k = p+q$. Then
$$
    * L\^{j} u = (-1)^{k(k+1)/2} L\^{n-j-k} \I u.
$$
\end{prop}

\begin{exam}
If $u$ is a primitive $(p,q)$--form, then this formula gives
$$
\lvert L\^{j} u \rvert^2 \omega\^{n}
= \tbinom{n-k}{j} \lvert u \rvert^2 \omega\^{n}
$$
after some manipulations. Working out the details of this is a fine way to
appreciate how error-prone these calculations become.
\end{exam}

If $u$ is a $(p,q)$--form with $p \leq q$, then we write $u =
\sum_{j=0}^p L\^{p-j} u_j$ for its primitive decomposition, where each
$u_j$ is a primitive $(j,j+q-p)$--form. This decomposition is
orthogonal, so
$$
\lvert u \rvert^2
= \sum_{j=0}^p \lvert L\^{p-j} u_j \rvert^2
= \sum_{j=0}^p \tbinom{n-2j-q+p}{p-j} \lvert u_j \rvert^2.
$$

\begin{prop}
    Let $u = \sum_{j=0}^p u_j \wedge \omega\^{p-j}$ be the primitive
decomposition of a $(p,q)$--form $u$, where $p \leq q$ and each $u_j$ is a
primitive $(j,j+q-p)$--form. Then
\begin{equation*}
  \Lambda\^l u 
  = \sum_{j=0}^{p-l} \tbinom{n-j-q+l}{l} L\^{p-j-l} u_j, 
\end{equation*}
and the decomposition of $\Lambda\^l u$ is primitive.
\end{prop}

\begin{proof}
  By linearity it is enough to prove this for a $(p,q)$--form
$u = L\^{p-j} u_j$, where $u_j$ is a primitive $(j,j+q-p)$--form. Let
$v$ be a form of degree $(n-p+l,n-q+l)$ and set $k = 2j+q-p$. Then
  \begin{align*}
    \la v, \Lambda\^l (L\^{p-j} u_j ) \ra \omega\^ n
    &= \la L\^l v, L\^{p-j} u_j \ra L\^ n \\
    &= L\^l v \wedge i^{q-p}(-1)^{\binom{k+1}{2}} L\^{n-j-q} \overline u_j  \\
    &= v \wedge i^{q-p}(-1)^{\binom{k+1}{2}}
    \tbinom{n-j-q+l}{l} L\^{n-j-q+l} \overline u_j  \\
    &= v \wedge * \bigl( \tbinom{n-j-q+l}{l} L\^{p-j-l} \overline u_j \bigr) \\
    &= \la v, \tbinom{n-j-q+l}{l} L\^{p-j-l} u_j \ra \omega\^ n.
  \end{align*}
Since the equality holds for all $v$, the result is proved for forms
of type $L\^{p-j} u_j$, where $u_j$ is primitive. Note that the
decomposition of $\Lambda\^l (L\^{p-j} u_j)$ is again primitive, so
the same will hold for an arbitrary form $u$.
\end{proof}

\begin{prop}
  \label{16}
  Let $u = \sum_{j=0}^p u_j \wedge \omega^{p-j}$ be the primitive
decomposition of a $(p,q)$--form $u$, where $p \leq q$ and $u_j$ is a
primitive $(j,j+q-p)$--form. Then
$$
u \wedge \overline{\I u} \wedge \omega^{n-p-q}
= \sum_{j=0}^p u_j \wedge \overline{\I u_j} \wedge \omega^{n-2j-q+p}.
$$
\end{prop}

\begin{proof}
  We induct on $p$, the result being clear for $p = 0$ or $q = 0$. If
$u$ is a $(p+1,q+1)$--form we have $u = \omega \wedge v_p + u_{p+1}$,
where $v_p$ is a $(p,q)$--form whose primitive decomposition is
evident. Then
$$
u \wedge \overline{\I u} = \omega^2 \wedge v_p \wedge \overline{\I v_p}
+ \omega \wedge v_p \wedge \overline{\I u_{p+1}}
+ \omega \wedge u_{p+1} \wedge \overline{\I v_p}
+ u_{p+1} \wedge \overline{\I u_{p+1}},
$$
so
$$
\displaylines{
u \wedge \overline{\I u} \wedge \omega^{n-p-q-2}
= v_p \wedge \overline{\I v_p} \wedge \omega^{n-p-q}
+ v_p \wedge \overline{\I u_{p+1}} \wedge \omega^{n-p-q-1}
\hfill\cr\hfill
{}+ u_{p+1} \wedge \overline{\I v_p} \wedge \omega^{n-p-q-1}
+ u_{p+1} \wedge \overline{\I u_{p+1}} \wedge \omega^{n-p-q-2}.
}
$$
The two middle terms are zero because $u_{p+1}$ is primitive, so $u_{p+1}
\wedge \omega^{n-p-q-1} = 0$ and $v_p \wedge \overline{\I v_p} \wedge
\omega^{n-p-q}$ is of the announced form by induction.
\end{proof}

\begin{rema}
  Here the reader may wonder what happens for a $(p,q)$--form with
$p+q > n$. The Lefschetz theorems tell all: That $L^k : \bw{n-k} V^* \to
\bw{n+k} V^*$ is an isomorphism entails that there are no primitive
$(p,q)$--forms with $p + q > n$, so no information is lost here by the wedge
product of such forms.
\end{rema}

\begin{prop}
    \label{prop:morphism}
Let $u = \sum_{j=0}^p u_j \wedge \omega\^{p-j}$ be the primitive
decomposition of a $(p,q)$--form $u$, where $p \leq q$ and each $u_j$ is a
primitive $(j,j+q-p)$--form. Then
$$
\displaylines{
    \Lambda\^l u \wedge \Lambda\^l \overline{\I u}
    \wedge \omega\^{n-p-q+2l}
    \hfill\cr\hfill{}
    = \sum_{j=0}^{p-l} 
    (-1)^j
    (-1)^{(q-p)(q-p+1)/2}
    \tbinom{p-j}{l}
    \tbinom{n-j-q+l}{p-l-j}
    \tbinom{n-j-q+l}{l}
    \lvert L\^{p-j} u_j \rvert^2
    \wedge \omega\^ n.
}
$$
\end{prop}

\begin{proof}
We first apply Proposition~\ref{16} to our $(p,q)$--form
$u = \sum_j L\^{p-j} u_j = \sum_j (u_j/(p-j)!) \wedge \omega^{p-j}$.
That gives
\begin{align*}
u \wedge \overline{\I u}
\wedge \omega\^{n-p-q} 
&= \sum_{j=0}^p \frac{1}{(n-p-q)!}
\Bigl( \frac{u_j\wedge \overline{\I u_j}}{(p-j)!^2} \Bigr) 
\wedge \omega^{n-2j-q+p}\\
&= \sum_{j=0}^p \frac{1}{(n-p-q)!}\frac{(n-2j-q+p)!}{(p-j)!^2} 
u_j \wedge \overline{\I u_j} \wedge \omega\^{n-2j-q+p}\\
&= \sum_{j=0}^p \tbinom{n-j-q}{p-j} \tbinom{n-2j-q+p}{p-j}
u_j \wedge \overline{\I u_j} \wedge \omega\^{n-2j-q+p}
\\
&= \sum_{j=0}^p 
(-1)^{\binom{2j+q-p+1}{2}} 
\tbinom{n-j-q}{p-j} \tbinom{n-2j-q+p}{p-j}
\lvert u_j \rvert^2 \omega\^{n}
\\
&= \sum_{j=0}^p (-1)^j(-1)^{(q-p)(q-p+1)/2}
\tbinom{n-j-q}{p-j} \tbinom{n-2j-q+p}{p-j}
\lvert u_j \rvert^2 \omega\^{n}.
\end{align*}
To get the general result, we apply this to the $(p-l,q-l)$--form
$\Lambda\^l u$. That gives (by letting $p \mapsto p - l$, $q \mapsto q
- l$)
$$
\displaylines{
  \Lambda\^l u \wedge \Lambda\^l \overline{\I u}
  \wedge \omega\^{n-p-q+l} 
  \hfill\cr\hfill
  \jot=0pt
  \eqalign{
  &= \sum_{j=0}^{p-l} 
  (-1)^{\binom{2j+q-p+1}{2}} 
  \tbinom{n-2j-q+p}{p-l-j}
  \tbinom{n-j-q+l}{p-l-j}
  \tbinom{n-j-q+l}{l}^2
  \lvert u_j \rvert^2
  \wedge \omega\^ n
  \cr
  &= \sum_{j=0}^{p-l} 
  (-1)^{\binom{2j+q-p+1}{2}} 
  \frac{
  \tbinom{n-2j-q+p}{p-l-j}
  \tbinom{n-j-q+l}{p-l-j}
  \tbinom{n-j-q+l}{l}^2
  }{\tbinom{n-2j-q+p}{p-j}}
  \lvert L\^{p-j} u_j \rvert^2
  \wedge \omega\^ n.
  }\!
}
$$
Once we remark that
$$
\frac{\binom{n-2j-q+p}{p-j-l} \binom{n-j-q+l}{l}}{\binom{n-2j-q+p}{p-j}}
= \tbinom{p-j}{l}
$$
the proof is finished.
\end{proof}

This last result is the key to proving what we want. It tells us how to
write the square of the norm of a $(p,q)$--form $u$ with $p \leq q$ as
a linear combination of traces of the form: Define two vector spaces
\begin{align*}
X &= \operatorname{Span}(
\Lambda\^l u \wedge \Lambda\^l \overline{\I u} 
\wedge \omega\^{n-p-q+2l} 
\mid l = 0,\ldots,p),
\\
Y &= \operatorname{Span}(|L\^{p-j}u_j|^2 \omega\^{n} \mid j=0,\ldots,p).
\end{align*}
(By a Zariski-open argument it is enough to prove our result on the
open set of forms $u$ where all the above symbols are nonzero.)
Proposition~\ref{prop:morphism} defines a linear morphism $A : X \to
Y$; a morphism that only depends on the dimension of $V$ and the
degree $p$, but is otherwise independent of the form $u$. Since
$|u|^2\omega\^n = \sum |L\^{p-j}u_j| \omega\^n$, the coefficients of
the linear combination we seek are the coordinates of the vector
$A^{-1}(1,\ldots,1)$. This observation shows that coefficients like
the ones we seek exist, the task is now to show that they coincide
with our integer sequence.

\begin{rema}
I'll say a little about how we originally found the main result of this
paper in case the reader is curious. First, we guessed that some kind of
linear combination like the one in Theorem~\ref{thm:norm} existed, but
assumed that its coefficients at least depended on the dimension of the
underlying vector space. Then we calculated our way to
Proposition~\ref{prop:morphism}. Once there, we calculated
$A^{-1}(1,\ldots,1)$ for $(1,1)$, $(2,2)$, $(3,3)$, $(4,4)$ and
$(5,5)$--forms with the help of computer algebra software, from which we
guessed that the coefficients were in fact independent of the vector space.
Searching the OEIS then revealed the coefficients probably formed a known
sequence, and from there it was not difficult to prove the main result.  
\end{rema}

\begin{proof}[Proof of Theorem~\ref{thm:norm}]
    If $u$ is a $(p,q)$--form with $p \leq q$ then we set $k = p+q$
and write
$$
\lvert u \rvert^2 \omega\^n
= \sum_{l=0}^n (-1)^{l + k(k+1)/2} b_l(p,n) \, 
\Lambda\^l u \wedge \Lambda\^l \overline{\I u}
\wedge \omega\^{n-p-q+2l},
$$
where $b_l(p,n)$ is the coefficient whose existence is guaranteed by
Proposition~\ref{prop:morphism}.  We will prove that $b_l(p,n) = b_l$ by
induction on $p$. 

We first remark that $b_0(p,n) = b_0 = 1$ for all $p, q, n$, because
the norm of a primitive $(p,q)$--form $u$ is $\lvert u \rvert^2 \omega\^n =
(-1)^{k(k+1)/2} u \wedge \overline{\I u} \wedge \omega\^{n-p-q}$.

For the induction step, we assume that $b_l(p,n) = b_l$ for $l = 0, \ldots,
p-1$ and want to prove that $b_{p}(p,n) = b_{p}$. For this, first
recall that if $u = \sum_j L\^{p-j} u_j$ is the primitive decomposition of
a $(p,q)$--form with $p \leq q$, then 
$$
\lvert u \rvert^2 \omega\^n 
= \sum_j \lvert L\^{p-j} u_j \rvert^2.
$$ 
Let's record for immediate use that if $p \leq q$ then
$(-1)^{(q-p)(q-p+1)/2} = (-1)^p(-1)^{(p+q)(p+q+1)/2}$. Then we can also write
the above as
$$
\def\bil{\mkern26mu}
\displaylines{
\lvert u \rvert^2 \omega\^n 
\hfill\cr\noalign{\vskip-3pt}\bil
{}
= \sum_{l=0}^n (-1)^{k(k+1)/2+l} b_l(p,n) \, 
\Lambda\^l u \wedge \Lambda\^l \overline{\I u}
\wedge \omega\^{n-2(p-l)}
\hfill\cr\bil
{}= \sum_{l=0}^n (-1)^{k(k+1)/2+l} b_l(p,n) 
\sum_{j=0}^l 
(-1)^{j}
(-1)^{(q-p)(q-p+1)/2}
\hfill\cr\hfill{}\times
  \tbinom{p-j}{l}
  \tbinom{n-j-q+l}{n-p-q+2l}
  \tbinom{n-j-q+l}{l}
  \lvert L\^{p-j} u_j \rvert^2
  \wedge \omega\^ n
  \phantom{.}
\cr\quad\hfill
{}= \sum_{l=0}^n (-1)^{l+p} b_l(p,n) \, 
\sum_{j=0}^l 
(-1)^j
  \tbinom{p-j}{l}
  \tbinom{n-j-q+l}{n-p-q+2l}
  \tbinom{n-j-q+l}{l}
  \lvert L\^{p-j} u_j \rvert^2
  \wedge \omega\^ n.
}
$$
By comparing the coefficients of $\lvert L\^p u_0 \rvert^2$ in these
two expressions we~find%
\begin{align*}
1 &= 
\sum_{l=0}^p (-1)^{l+p} b_l(p,n) 
\, 
\tbinom{p}{l}
\tbinom{n-q+l}{n-p-q+2l}
\tbinom{n-q+l}{l}
\\
&= 
\tbinom{n}{p} b_p(p,n)
+ \sum_{l=0}^{p-1} 
(-1)^{l+p}
b_l \, 
\tbinom{p}{l}
\tbinom{n-q+l}{p-l}
\tbinom{n-q+l}{l}
\end{align*}
for all $n \geq p+q$. The binomial coefficient $\binom{n}{k}$ is a
polynomial of degree $k$ in $n$ whose leading term is
$1/k!$. Comparing the top-degree coefficients of $n$ in the above
equation we find that
$$
0 =
\frac{1}{p!} b_p(p,n)
+ \sum_{l=0}^{p-1} (-1)^{l+p} b_l \, 
\tbinom{p}{l}
\frac{1}{l! (p-l)!}.
$$
Since this equation expresses $b_p(p,n)$ in terms of things that do not
depend on $n$, we conclude that $b_p$ doesn't depend on $n$ either.
The defining recurrance relation \eqref{bezel} for the integers $b_l$,
now shows that $b_p(p,n) = b_p$.

Finally, we remark that by conjugating the form $u$ it is enough
to prove our formula for forms $u$ with $p \leq q$.
\end{proof}

\begin{coro}
If $u$ and $v$ are complex $(p,q)$--forms on $V$, then
$$
(-1)^{k(k+1)/2}
\la u, v \ra \, \omega\^n
= \sum_{l=0}^{n} 
(-1)^{l} b_l \, 
\Lambda\^l u \wedge \Lambda\^l \overline{\I v} \wedge \omega\^{n-p-q+2l}.
$$
\end{coro}

\begin{proof}
  Immediate from polarization.
  \vadjust{\penalty-2000}
\end{proof}

\begin{exam}
In addition to the well-known formula for the square of the norm of a
$(1,1)$--form, Theorem~\ref{thm:norm} gives these formulas for the norms
of higher-degree forms, that have not appeared before to the best of my
knowledge. We a few here, for real forms to simplify notation.

\smallskip
\noindent
(i)\quad
For a real $(2,2)$--form $u$ on $V$ we have
$$
|u|^2 \omega\^{n}
= u^2 \wedge \omega\^{n-4}
- (\Lambda u)^2 \wedge \omega\^{n-2}
+ 3 (\Lambda\^{2} u)^2 \wedge \omega\^{n}.
$$
    
\smallskip
\noindent
(ii)\quad
For a real $(3,3)$--form $u$ our formula gives
$$
|u|^2 \omega\^{n}
= 
- u^2 \wedge \omega\^{n-6}
+ (\Lambda u)^2 \wedge \omega\^{n-4}
- 3 (\Lambda\^{2} u)^2 \wedge \omega\^{n-2}
+ 19 (\Lambda\^{3} u)^2 \wedge \omega\^{n}.
$$

\smallskip
\noindent
(iii)\quad
For a real $(4,4)$--form $u$ we get
$$
\displaylines{
|u|^2 \omega\^{n}
= 
 u^2 \wedge \omega\^{n-8}
- (\Lambda u)^2 \wedge \omega\^{n-6}
+ 3 (\Lambda\^{2} u)^2 \wedge \omega\^{n-4}
\hfill\cr\hfill
{}- 19 (\Lambda\^{3} u)^2 \wedge \omega\^{n-2}
+ 211 (\Lambda\^{4} u)^2 \wedge \omega\^{n}.
}
$$
\end{exam}

Our theorem allows us to express the scalar product of two forms as a
wedge product of forms derived from the original ones. Doing things the
other way around, or expressing a wedge product in terms of inner products
is also possible:

\begin{coro}
    If $u$ and $v$ are $(p,q)$--forms on $V$ and $k = p+q$, then
    $$
    (-1)^{k(k+1)/2} u \wedge \overline{\I v} \wedge \omega\^{n-k}
    = \sum_{m = 0}^n (-1)^{m} 
    \la \Lambda\^m u , \Lambda\^m \overline{\I v} \ra \, \omega\^{n}.
    $$
\end{coro}

\begin{proof}
We remark that as usual it is enough to prove our statement for
$(p,q)$--forms with $p \leq q$, so we assume this holds. Plugging
$\Lambda\^m u$ and $\Lambda\^m \I v$ into our formula gives
$$
\displaylines{
(-1)^{k_m(k_m+1)/2} 
\la \Lambda\^m u , \Lambda\^m \overline{\I v} \ra \, \omega\^{n}
\hfill\cr\hfill
{} = \sum_{l=0}^{n} 
(-1)^l b_l \, 
\tbinom{l+m}{l}^2
(\Lambda\^{l+m} u) \wedge (\Lambda\^{l+m} \overline{\I v}) 
\wedge \omega\^{n-k+2(l+m)},
}
$$
where we write $k_m = k-2m$. Remark that 
$$
(-1)^{k_m(k_m+1)/2} =
(-1)^{k(k+1)/2} (-1)^m.
$$  
If we sum both sides of the above equation for the scalar product over $m$
from $0$ to $n$ and then change the variable in the first sum from $m$ to
$\nu = l+m$ we get 
$$
\displaylines{
    (-1)^{k(k+1)/2} 
    \sum_{m=0}^n
    (-1)^{m} \la \Lambda\^m u , \Lambda\^m \overline{\I v} \ra \,
    \omega\^{n}
    \hfill\cr\noalign{\vskip-5pt}\hfill
    \jot=0pt
    \eqalign{
    {} &= 
    \sum_{\nu=0}^{n} 
    \Bigl(
    \sum_{l=0}^{\nu}
    (-1)^l b_l \, 
    \tbinom{\nu}{l}^2
    \Bigr)
    (\Lambda\^{\nu} u) \wedge (\Lambda\^{\nu} \overline{\I v}) 
    \wedge \omega\^{n-k+2\nu}
    \cr\noalign{\vskip3pt}{}
    &= u \wedge \overline{\I v} \wedge \omega\^{n-k},
}
}
$$
because $\sum_{l=0}^{\nu} (-1)^l b_l \, \tbinom{\nu}{l}^2 = 0$ for all $\nu \geq 1$ by definition.
\end{proof}

\begin{rema}
Our formula for the inner product is given by numbers related to Bessel
functions. It is possible to write our results in compact form by letting
holomorphic functions define sesquilinear operators on $V$. For this, let 
$$
f(z) = J_0(2\sqrt z) = \sum_{m\geq0} (-1)^m \frac{1}{m!^2} z^m.
$$
This is a Bessel function of the first kind and if we set $z = xy$ and look at its reciproque we find
$$
\frac{1}{f(xy)} = \sum_{l\geq 0} (-1)^l b_l\, x\^l y\^l;
$$
see \cite{Carlitz,Riordan}. This function defines a sesquilinear operator
on $\bw{*} V^*$ if we declare that 
$$
x\^a y\^b(\Lambda,\Lambda) 
:= (u,v) \mapsto \Lambda\^a u \wedge \Lambda\^b \overline v.
$$ 
We can then write
$$
\eqalign{
(-1)^{k(k+1)/2} u \wedge \overline{\I v} \wedge \omega\^{n-p-q}
    &= 
    \pi_{n,n}\Bigl(
    f(xy)(\Lambda,\Lambda)(u,\overline{\I v})
    \wedge \exp(\omega)
    \Bigr),
    \cr
(-1)^{k(k+1)/2}
\la u,v \ra \, \omega\^n
&= 
\pi_{n,n}\Bigl(
\frac{1}{f(xy)}(\Lambda,\Lambda)(u,\overline{\I v})
\wedge \exp(\omega)
\Bigr),
}
$$
where $\pi_{n,n} : \bw{*} V^* \to \bw{n,n} V^*$ is the 
projection onto the subspace of $(n,n)$--forms and $\exp(\omega) :=
\sum_{k\geq0} \omega\^k$. In applications we would most likely consider
differential forms on a manifold $X$ and be interested in the global inner
product $\langle\!\langle u,v \rangle\!\rangle = \int_X \langle u,v\rangle \,
\omega\^n$ instead of the pointwise one. There integration 
kills all but the top-degree forms, so we can write
$$
(-1)^{k(k+1)/2} \langle\!\langle u,v \rangle\!\rangle 
= \int_X \frac{1}{f(xy)}(\Lambda,\Lambda)(u,\overline{\I v})
\wedge \exp(\omega).
$$

Consider now the set $U$ of complexified Hermitian inner products on $V$,
that is the set of $(1,1)$--forms $\alpha + i\omega$, where $\alpha$ and
$\omega$ are real and $\omega$ is an inner product. We have a trivial
holomorphic vector bundle $F^k \to U$ whose fiber is the space of
$k$--forms on $V$, and the inner products $\omega$ define a Hermitian
metric $h$ on $F^k$. The words ``mirror symmetry'' may be waived
around here in certain crowds (but then we should perhaps look at $G^k
\!=\!\!
\smash{\bigoplus\limits_{p - q = n - k} \bw{p,q} V^*}$ instead of $F^k$).

It is tempting to use the differential equation that $J_0$ satisfies
to say something about the curvature tensor of $h$, but in practice this
seems difficult at best and should in fact depend heavily on how the space
of primitive $k$--forms varies with $\omega$ inside $\bw{k}V^*$. The
preprint \cite{HuyArx} may be useful here.
\end{rema}

\section{Curvature tensors and the Kobayashi--L\"{u}bke inequality}
\label{sec:tw}

\noindent
\textbf{Linear algebraic preliminaries}
\hskip1em
In this section, we will (morally speaking) be viewing the curvature form
$R$ as a form defined on the total space of its vector bundle $E$ and
equipping that space with a metric induced by the ones on the underlying
space $X$ and on $E$. This lets us use the Hodge star and Lefschetz
operators on that bigger space and apply the results from
Section~\ref{sec:on}.

Let $V$ and $E$ be complex vector space of dimensions $n$ and $r$, equipped
with Hermitian metrics $\omega$ and $h$. Let $R$ be a curvature-type
tensor, or an element of $\bigwedge^{1,1} V^* \otimes \End E$
that is Hermitian. We view $R$ as a $(2,2)$--form on the space $E \oplus
V$, equipped with the Hermitian metric $\alpha = \omega + h$, where we
abuse notation and do not write $\alpha = p_V^*\omega + p_E^*h$ as we should.

Let $e$ be a form on $E$ and $v$ a form on $V$. By picking orthonormal
coordinates we quickly verify that
$$
*_\alpha(p_V^* v \wedge p_E^* e) 
= p_V^*(*_\omega v) \wedge p_E^*(*_h e).
$$
Since the exterior algebra of $V \oplus E$ is generated by elements of the
type $p_V^* v \wedge p_E^* e$ this lets us calculate with the Hodge star
operator on that space.

Let $\Lambda_\omega$ and $\Lambda_h$ be the adjoints of the Lefschetz
operators of $\omega$ and $h$, pulled back to $V \oplus E$. These
operators commute by general facts on trace operators in
Coffman's~\cite{Coffman} or by calculations in an orthonormal basis.

We also note that Newton's binomial formula gives
$$
\alpha\^{l} = \sum_{k=0}^{l} \omega\^{k} \wedge h\^{l-k}
$$
and that many of those terms will be zero for $l$ big for degree
reasons. A similar formula expresses $\Lambda\^{l}_\alpha$ in terms of
$\Lambda\^{k}_\omega$ and $\Lambda\^{l-k}_h$.

Finally we set $k! c_k := \Lambda\^{k}_h (\bigwedge^k \! R)$ for $k = 0, \ldots,
r$.  The notation is so chosen because when $R$ is the curvature tensor of
an actual Hermitian metric on a vector bundle, the $c_k$ will be the Chern
forms defined by $R$. The $k!$ factor deserves an explanation:

The inner product $h$ is an isomorphism $h : E \to \overline E^*$. It
induces inner products on both $\End E$ and $\bigwedge^{1,1} \! E^*$ and a
morphism $h  \otimes \id_{E^*}: \End E \to \bigwedge^{1,1} \!
E^*$. The trace of an endomorphism of $E$ is just its scalar product
again the identity morphism.  Taking $k$--th exterior powers we get a
canonical morphism
$h^k \otimes \id_{\wedge^k E^*} : \End \bigwedge^k \! E \to
\bigwedge^{k,k} \! E^*$. This morphism is however not an isometry, but
$h\^k$ is; this can be seen by comparing the norms of $\id_{\wedge^k
E}$ and its image $h\^k$. We now want to find a morphism
$\bigwedge^{1,1} \! E \to \bigwedge^{k,k} \! E$ that makes the diagram
$$
\xymatrix@C+10pt{
    \End E \ar[r]^-{f \mapsto \wedge^k f} \ar[d]^{h \otimes \id_{E^*}} & 
    \End \bigwedge^{k} \! E \ar[d]^{h\^k \otimes \id_{\wedge^k E^*}} \\
    \bigwedge^{1,1} \! E \ar@{-->}[r] 
    & \bigwedge^{k,k} \! E
}
$$
commute. This morphism is clearly $u \mapsto u^k / k!$, whence the factor
of $k!$~above.%
\vadjust{\penalty-2000}

\smallskip
\noindent
\textbf{The norm of a curvature tensor}
\hskip1em
Here we calculate the norm of a curvature tensor of a vector
bundle. The identity we find is implicit in the literature on the
Kobayashi--L\"{u}bke inequality (compare with \cite{Chen-Ogiue},
\cite{Lubke} and \cite{Siu}); the inequality is actually a corollary of a
simple application of Cauchy--Schwarz to the equation for the norm of the
curvature tensor.

\begin{theo}
    Let $E \to X$ be a holomorphic vector bundle of rank $r$ over a complex
manifold $X$ of dimension $n$. Let $\omega$ and $h$ be Hermitian metrics on
$X$ and $E$, respectively. Let $\curv$ be the curvature form of $(E,h)$
and let $c_k$ be the Chern forms defined by the curvature form. Then
$$
\Bigl\lvert \curv \Bigr\rvert^2 \omega\^n
= (2 c_2 - c_1^2) \wedge \omega\^{n-2}
+ \Bigl\lvert \tr_\omega \curv \Bigr\rvert^2 \omega\^n
$$
at every point of $X$. If $(E,h)$ is Hermite--Einstein, then we also have
$$
0 \leq 
(2r c_2 - (r-1) c_1^2) \wedge \omega\^{n-2}
$$
pointwise on $X$ with equality if and only if $\curv = (\lambda/n) \id_E
\otimes\, \omega$, where $\lambda$ is the Hermite--Einstein constant of
$(E,h)$.
\end{theo}

\begin{proof}
   The announced result is local on $X$, so we pick a point $x \in X$ and
write $V = T_{X,x}$, abuse notation to write $E = E_{x}$ and write $R$
for the image of $\curv$ under the isometry $\bigwedge^{1,1} \! T_X
\otimes \End E \to \bigwedge^{2,2} (T_X \otimes E)^*$ defined by $h$. Then
$R$ is a $(2,2)$--form on $V \oplus E$.
We write $\alpha = \omega + h$ for the induced inner product on
$V \oplus E$, in slight abuse of notation.

The norm of $R$ as a $(2,2)$--form on $V \oplus E$ is
$$
|R|^2 \alpha\^{n+r}
= R^2 \wedge \alpha\^{n+r-4}
- (\Lambda_\alpha R)^2 \wedge \alpha\^{n+r-2}
+ 3 (\Lambda\^{2}_\alpha R)^2 \wedge \alpha\^{n+r}.
$$
We'll indicate the general steps in the calculation of each of these
factors but leave the details mostly to the reader.  We have
\begin{align*}
R^2 \wedge \alpha\^{n+r-4}
&= R^2 \wedge \biggl(\sum_{k=0}^4\omega\^{n-k} \wedge h\^{r+k-4}\biggr)
\\
&= R^2 \wedge \omega\^{n-2} \wedge h\^{r-2}
= 2 c_2 \wedge \omega\^{n-2} \wedge h\^{r},
\end{align*}
where the second equality holds for degree reasons. Similarly we get
$$
\displaylines{
(\Lambda_\alpha R)^2 \wedge \alpha\^{n+r-2}
= (\Lambda_\omega R)^2 \wedge \omega\^{n} \wedge h\^{r-2} 
\hfill\cr\hfill{}
+ 2 (\tr_\omega c_1)^2 \omega\^{n} \wedge h\^{r}
+ c_1^2 \wedge \omega\^{n-2} \wedge h\^{r} 
}
$$
because 
$$
\Lambda_\omega R \wedge \Lambda_h R \wedge \omega\^{n-1} \wedge h\^{r-1}
= (\tr_\omega c_1)^2 \omega\^{n} \wedge h\^{r}.
$$
Finally, 
$$
(\Lambda\^{2}_\alpha R)^2 \wedge \alpha\^{n+r} 
= (\Lambda_\omega \Lambda_h R)^2 \wedge \omega\^{n} \wedge h\^{r}
= (\tr_\omega c_1)^2 \wedge \omega\^{n} \wedge h\^{r},
$$
again for degree reasons and commutativity of the adjoints of the Lefschetz
operators. From this we reap
$$
\displaylines{
    \lvert R \rvert^2 \omega\^{n} \wedge h\^{r}
    = (2 c_2 - c_1^2) \wedge \omega\^{n-2} \wedge h\^{r} 
    \hfill\cr\hfill{}
    + (\tr_\omega c_1)^2 \wedge \omega\^{n} \wedge h\^{r}
    - (\Lambda_\omega R)^2 \wedge \omega\^{n} \wedge h\^{r-2}
}
$$
which yields
$$
    \lvert R \rvert^2 \omega\^{n}
    = (2 c_2 - c_1^2) \wedge \omega\^{n-2}
    + (\tr_\omega c_1)^2 \, \omega\^{n}
    - \Lambda\^{2}_h(\Lambda_\omega R)^2 \, \omega\^{n}.
$$
We now use the formula for the norm of a $(1,1)$--form and see that
$$
\lvert \Lambda_\omega R \rvert_h^2
= (\Lambda_h\Lambda_\omega R)^2 - \Lambda\^{2}_h(\Lambda_\omega R)^2
= (\tr_\omega c_1)^2 - \Lambda\^{2}_h(\Lambda_\omega R)^2,
$$
thus obtaining our first announced result (in equivalent notation):
$$
    \lvert R \rvert^2 \, \omega\^{n}
    = (2 c_2 - c_1^2) \wedge \omega\^{n-2}
    + \lvert \Lambda_\omega R \rvert_h^2 \, \omega\^{n}.
$$

Now assume that $(E,h)$ is Hermite--Einstein. By definition, this means
that $\tr_\omega \curv = \lambda \id_{E}$. Under our isometries, this
translates into $\Lambda_\omega R = \lambda h$. The factor $\lambda$
satisfies $r \lambda = \tr_\omega c_1$, so we get
\begin{align*}
    0 \leq 
    \lvert R \rvert^2 \, \omega\^{n}
    &= (2 c_2 - c_1^2) \wedge \omega\^{n-2}
    + r |\lambda|^2 \, \omega\^{n}.
    \\
    &= (2 c_2 - c_1^2) \wedge \omega\^{n-2}
    + \tfrac 1r (\tr_\omega c_1)^2 \omega\^{n}.
\end{align*}
Multiplying by $r$ and rearranging gives 
$$
0 \leq
r \lvert R \rvert^2 \, \omega\^{n}
= (2r c_2 - (r-1)c_1^2) \wedge \omega\^{n-2}
+ \lvert c_1 \rvert^2 \omega\^{n}.
$$
Proposition~\ref{prop:CS} below, which is just the Cauchy--Schwarz
inequality in disguise, says that
$$
\lvert c_1 \rvert^2 \leq r \lvert R \rvert^2,
$$
with equality if and only if $R = u \wedge \omega$, where $u$ is the
pullback of a form on $E$. By the Hermite--Einstein condition we
necessarily have $u = (\lambda/n) h$ in that case. This proves the
Kobayashi--L\"{u}bke inequality.
\end{proof}

\begin{prop}
\label{prop:CS}
We have $\lvert \Lambda_\omega R\rvert^2 \leq n \lvert R \rvert^2$ and
$\lvert c_1 \rvert^2 = \lvert \Lambda_h R\rvert^2 \leq r \lvert R
\rvert^2$, with equalities if and only if $R = u \wedge \omega$ or $R = v
\wedge h$, where $u$ and $v$ are pullbacks of $(1,1)$--forms from $E$ and
$V$, respectively. 
\end{prop}

\begin{proof}
    We just prove the first result since the proof of the second
differs from that in notation only.  The primitive decomposition of
$R$ as a $(2,2)$--form on $E \oplus V$ is 
$$
R = r_0 \omega \wedge h + r_1^\omega \wedge h + r_1^h \wedge \omega
+ r_2.
$$
Here $r_0$ is a scalar, $r_1^h$ is a primitive form that's a pullback from
$E$, similar for $r_1^\omega$. By orthogonality we see that $\Lambda_\omega
r_j = 0$ for these primitive forms. This gives 
$\Lambda_\omega R = n (r_0 h + r_1^h)$,
so
\begin{align*}
\lvert \Lambda_\omega R \rvert^2 
&= n^2 (\lvert r_0 h\rvert^2 +\lvert r_1^h\rvert^2)
= n (\lvert r_0 \omega \wedge h\rvert^2 
+ \lvert r_1^h \wedge \omega \rvert^2)
\\
&\leq
n(\lvert r_0 \omega \wedge h \rvert^2 
+ \lvert r_1^h \wedge \omega 
+ r_1^\omega \wedge h \rvert^2 
+ \lvert r_2 \rvert^2 )
= n \lvert R \rvert^2
\end{align*}
with equality if and only if $R = u \wedge \omega$ for a form $u$
that's a pullback from~$E$.%
\end{proof}

\begin{exam}
(1) Let $(X,\omega)$ be a K\"ahler--Einstein manifold, so that $\Ric \omega
= \lambda \omega$ for some $\lambda \in \RR$. Then%
$$
\lvert R \rvert^2 \omega\^n
= c_2 \wedge \omega\^{n-2} + (n \lambda)^2 \omega\^n,
$$
where $c_2$ is the second Chern form defined by $\omega$.

\smallskip
(2) Let $(X,\omega)$ now be a K\"ahler manifold of constant sectional
curvature $\lambda$; like projective space with the Fubini--Study
metric, a torus with its flat metric or the unit ball with the Bergman
metric. Then the curvature tensor of $\omega$ is
$$
R_{jklm} = \lambda(\omega_{jk}\omega_{lm} - \omega_{jl}\omega_{mk})
$$
in local coordinates and we have $\Ric \omega =
\lambda(n-1)\,\omega$. Some calculations give
$$
|R|^2 
= 2n(n-1) \lambda^2,
$$
so we see that
$$
c_2 \wedge \omega\^{n-2}
= -\lambda^2(n-2)(n-1)n(n+1)\,\omega\^n.
$$
\end{exam}

\begin{rema}
The original motivation for all of this was that I didn't understand
where the differential-geometric proofs of the Kobayashi--L\"{u}bke
inequality came from (see \cite{Chen-Ogiue,Lubke,Siu}), since they all
brutally calculate things in local coordinates. I also naively thought that
if I found a more coordinate-invariant proof of the inequality it would be
possible to use it to find inequalities involving higher Chern classes,
because if calculating $|\curv|^2$ gives an inequality involving $c_2$ then
calculating $|\bw{k}\curv|^2$ should give an inequality involving
$c_{2k}$. Unfortunately this does not seem to be possible, basically
because we cannot calculate $\Lambda(u \wedge v)$ or $*(u \wedge v)$
in terms of $u$, $v$ and $\Lambda$ or $*$, which again is not possible
because the wedge product of primitive forms is not primitive. I leave to
the reader the pleasure of trying to estimate $|\bw{k}\curv|^2$ in terms of
things we know and are interested in and seeing where things go wrong.
\end{rema}

\bibliographystyle{amsalpha}
\bibliography{cci}

\end{document}